\documentclass[12pt]{article}
\usepackage{amsmath, amsfonts, amsthm, amssymb, color, hyperref, extarrows, enumitem}

\newtheorem{theorem}{Theorem}[section]

\newtheorem{definition}[theorem]{Definition}
\newtheorem{cor}[theorem]{Corollary}
\newtheorem{lem}[theorem]{Lemma}
\newtheorem{pro}[theorem]{Proposition}

\numberwithin{equation}{section}

\usepackage[T1]{fontenc}

\usepackage[titletoc]{appendix}

	\newcommand{\abs}[1]{\left|#1\right|}

\usepackage{cite}
 
\hypersetup{
	colorlinks   = true,
	citecolor    = magenta}

	\usepackage[titletoc]{appendix}

\begin{document}

\title{\vspace{-1.2cm} \bf Bergman-Calabi diastasis and K\"ahler metric of constant holomorphic sectional curvature \rm}
\author{Robert Xin Dong \quad  and  \quad  Bun Wong}

\date{\it \small Dedicated to Professor Joseph J. Kohn}
\maketitle

\begin{abstract}
We prove that for a bounded domain in $\mathbb C^n$ with the Bergman metric of constant holomorphic sectional curvature being biholomorphic to a ball is equivalent to the hyperconvexity or the exhaustiveness of the Bergman-Calabi diastasis. By finding its connection with the Bergman representative coordinate, we give explicit formulas of the Bergman-Calabi diastasis and show that it has bounded gradient. In particular, we prove that any bounded domain whose Bergman metric has constant holomorphic sectional curvature is Lu Qi-Keng. We also extend a theorem of Lu towards the incomplete situation and characterize pseudoconvex domains that are biholomorphic to a ball possibly less a relatively closed pluripolar set.

\end{abstract}

\renewcommand{\thefootnote}{\fnsymbol{footnote}}
\footnotetext{\hspace*{-7mm} 
\begin{tabular}{@{}r@{}p{16.5cm}@{}}
& Keywords.  Bergman metric, Bergman representative coordinate, holomorphic sectional curvature, hyperconvex domain, Lu Qi-Keng domain, $L^2$-domain of holomorphy, pluripolar set\\
& Mathematics Subject Classification. Primary 32F45; Secondary 32T05, 32Q05, 32D20\\

\end{tabular}}

\section{Introduction}

In \cite{Lu}, Lu proved his well-known uniformization theorem: a bounded domain in $\mathbb C^n$ with a complete Bergman metric of constant holomorphic sectional curvature is biholomorphic to the Euclidean ball.  
For  a bounded domain  $\Omega \subset \mathbb C^n$, 
fix a point $z_0\in \Omega$ and let $A_{z_0}:=\{z\in \Omega \, | \,K(z, z_0) =0 \, \}$ be the zero set  of the Bergman kernel $K(\cdot , z_0)$. Since $A_{z_0}$ is an analytic variety, as domains $\Omega \setminus A_{z_0}$ and $\Omega$ have the same Bergman kernel $K$ and Bergman metric $g$.
Consider on $\Omega \setminus A_{z_0}$ the K\"{a}hler potential
\begin{equation} \label{dia}
\Phi_{z_0}(z):= \log \frac{K(z, z)  K(z_0, z_0)}{ |K(z, z_0)|^2} 
\end{equation}
for the Bergman metric $g=\partial \overline  \partial   \Phi_{z_0}$. 
  Locally, the right hand side of \eqref{dia} coincides with Calabi's diastasis \cite{Ca}.  We call the function $\Phi_{z_0}(z)$ the Bergman-Calabi diastasis relative to $z_0$.
 In this paper, we shall prove the following main theorem.

\begin{theorem} \label{biholo} Let $\Omega \subset \mathbb C^n$ be a bounded domain whose Bergman metric has its holomorphic sectional curvature  identically equal to a negative constant $-c^2$. Then the followings are equivalent:
\begin{enumerate}
\item [(i)] $\Omega$ is Bergman complete; 

\item [(ii)] for some  $z_0\in \Omega$, $\Phi_{z_0}(z)$ blows up to infinity at the boundary of $\Omega$;    

\item [(iii)] $\Omega$ is hyperconvex;

\item [(iv)] $\Omega$ is biholomorphic to the Euclidean ball $\mathbb B^n$ and $n = 2/c^2-1$.
   
\end{enumerate}

\end{theorem}

 \medskip

 \noindent{}{\bf Remarks.}

\begin{enumerate}[label=(\roman*)]
 
\item [(a)] The equivalence between (i) and (iv) is exactly Lu's theorem \cite{Lu}. A domain is said to be Bergman complete if it is a complete metric space with respect to the distance induced by the Bergman metric.

\item [(b)] 

It follows from the transformation rule of the Bergman kernel that the Bergman-Calabi diastasis is invariant under biholomorphic mappings in the sense that if $f$ is a biholomorphic map from $\Omega_1$ to $\Omega_2$, then
$$
\Phi_{\Omega_1; z_0}(z)={\Phi}_{\Omega_2; f(z_0)}(f(z)).
$$
Thus, Condition (ii) is preserved under biholomorphic mappings. For general bounded domains, as demonstrated by our examples in Proposition \ref{annulus}, Condition (ii) does not imply the Bergman completeness.

\item [(c)]

Under the constant negative holomorphic sectional curvature assumption in Theorem \ref{biholo}, Condition (ii) in fact implies  that $\Omega$ is Bergman exhaustive, namely its Bergman kernel function $K(z, z)$ blowing up to infinity at any boundary point (see Proposition \ref{Mok-Yau}). When $n=1$, Chen in \cite{C00} proved that if a bounded domain in $\mathbb C$ is Bergman exhaustive then it is Bergman complete;  moreover, the converse is not true as shown by Zwonek \cite{Zw}. When $n\geq 2$, the Bergman exhaustiveness and completeness do not imply each other in general (see \cite[Chap. 15]{JP}).  
For example, the Hartogs triangle $\mathcal H:=\{z \in \mathbb C^2 : |z_1|<|z_2|<1\}$ is Bergman exhaustive. 
But from (b) we know that $\mathcal H$ does not satisfy Condition (ii) and is not Bergman complete, as it
is biholomorphic to $\mathbb D \times \mathbb D^* $.
 
 The motivation of Condition (ii) in Theorem \ref{biholo} comes from the work \cite{CW} of Cheung and the second author who proved that if a bounded convex domain in $\mathbb C^n$ with a Hermitian metric of constant negative holomorphic sectional curvature such that all metric components blow up to infinity at the boundary, then the domain is biholomorphic to a  ball (see also \cite{CW2, W}). 

\item [(d)] A bounded domain is said to be hyperconvex if there exists a continuous negative plurisubharmonic exhaustion function. Ohsawa in \cite{O93} proved that a bounded hyperconvex domain is Bergman exhaustive. B\l{}ocki and Pflug in \cite{BP98} and Herbort in \cite{H99} independently proved that a bounded hyperconvex domain is Bergman complete. On the other hand, by the works of Diederich and Forn\ae ss \cite{DF77}, Kerzman and Rosay \cite{KR81} and Demailly \cite{De87}  it is known that any pseudoconvex domain with $C^1$-smooth boundary is hyperconvex. Previously, it was shown that any pseudoconvex domain with $C^1$-smooth boundary is both Bergman exhaustive (see \cite{P75}) and Bergman complete (see \cite{O81}). In this paper, instead of imposing boundary regularity conditions, we place curvature conditions on the domain.

\item [(e)] 
 Theorem \ref{biholo} says that under the constant negative holomorphic sectional curvature assumption, the Bergman completeness, exhaustiveness of the Bergman-Calabi diastasis, and  hyperconvexity are all equivalent to one another, and either of them is equivalent to the domain being biholomorphic to a ball. Our proof of Theorem \ref{biholo}
is carried out as 
\begin{equation} \label{circle}
 (ii)  \Longrightarrow (iii)  \Longrightarrow (i)  \Longrightarrow  (iv)   \Longrightarrow (ii).
\end{equation}

The second and third implications follow from \cite{BP98, H99} and \cite{Lu}, respectively. We prove the first and last implications of \eqref{circle} in Section 3, where we additionally give a direct proof of the fact that $(ii) \Longrightarrow (iv)$. Therefore, the equivalence between $(ii) $ and $ (iv)$ can be proved straightforwardly without using Lu's theorem.

\end{enumerate}

 Unlike the complete K\"{a}hler-Einstein metric which is known to exist on any bounded pseudoconvex domain as shown by Cheng and Yau \cite{CY} and Mok and Yau \cite{MY83}, the Bergman metric is incomplete in many cases. To find more applications in K\"{a}hler geometry, our second motivation is to give an extension of Lu's theorem to a wider class of domains without the completeness assumption. 
Let $\Omega \subset \mathbb C^n$ be a bounded domain with the Bergman metric $g$. Recall that $\Omega$ is called a Lu Qi-Keng domain if for any  $p \in \Omega$, its Bergman kernel $K(\cdot , p)$  has no zero set. At $p\in  \Omega$, the Bergman representative coordinate $T (z) =(w_1, ... , w_n)^{\tau}$ is defined as
\begin{equation} \label{rep}
w_{\alpha} (z):=\sum _{j=1}^{n} g^{\bar j \alpha }(p) \left(K(z, p) ^{-1} \left. \frac{\partial}{\partial \overline {t_j} } \right|_{t=p} K(z, t)  -    \left. \frac{\partial}{\partial \overline {t_j} } \right|_{t=p}  \log K(t, t)\right),
\end{equation}
where $(g^{\bar j \alpha })  = (g_{\alpha \bar{j}})^{-1}$. It is well known that $T (z)$ is holomorphic on $\Omega $ less the zero set of $K(\cdot , p)$.

\medskip

  Our second theorem shows that a bounded domain is Lu Qi-Keng if its Bergman metric has constant holomorphic sectional curvature, and in this case the Bergman representative coordinate $T$ maps $\Omega$ to a ball. Moreover, we give explicit formulas for the Bergman-Calabi diastasis, whose gradient is proved to be bounded.

 \begin{theorem} \label{zero} Let $\Omega \subset \mathbb C^n$ be a bounded domain whose Bergman metric $g$ has its holomorphic sectional curvature  identically equal to a negative constant $-c^2$. Then for any $p \in \Omega$ it holds that
 \begin{enumerate}

\item [1)] the Bergman kernel $K(\cdot , p)$ has no zero set;
 
\item [2)] $T$ defined by \eqref{rep} maps $\Omega$ to
$ \{(w_1, ... , w_n)^{\tau} :   \sum_{\alpha, \beta=1}^n  w_\alpha   g_{ \alpha \bar \beta } (p) \overline {w_\beta }   < {2}{c^{-2}} \}$;

\item [3)] the Bergman-Calabi diastasis relative to $p$ can be written as
\begin{equation}  \label{Lu_g}
\Phi_{p}( z) =  {\frac{-2}{c^2}} \log      \left (1 - \frac{c^2}{2} \sum_{\alpha, \beta=1}^n  w_\alpha(z)   g_{ \alpha \bar \beta } (p) \overline {w_\beta(z)} \right ), \quad z \in  \Omega;
\end{equation}  

\item [4)] for any $z_0 \in \Omega$,
the length of $\partial \Phi_{z_0}$  measured by $g$  
is less than $\sqrt{2} |c|^{-1}$, namely, 
$$
 |\partial    \Phi_{z_0}|^2_g (p) < {2}{c^{-2}}.
$$

 \end{enumerate}

\end{theorem}

Previously, Lu's theorem in \cite{Lu} yields the above  conclusions 1) -- 4) under the additional assumption that $\Omega$ is Bergman complete. Theorem \ref{zero} also says that the map $T$ is holomorphic on $\Omega$.
However, if the completeness assumption in Lu's theorem
is dropped completely, then one cannot expect the same conclusion as what he proved, namely the domain being necessarily biholomorphic to the ball. In fact, if $E$ is a non-empty relatively closed pluripolar subset of $\mathbb B^n$, then any domain in $\mathbb C^n$ that is biholomorphic to $\mathbb B^n  \setminus E$ admits an incomplete Bergman metric of constant holomorphic sectional curvature. We say a set $E$ is pluripolar if there exists a plurisubharmonic function $\varphi$ in $\mathbb C^n$ such that $\varphi = -\infty$ on $E$, and a result \cite{S82} of Siciak  implies that the  Bergman spaces on $\mathbb B^n  \setminus E$ and on $\mathbb B^n$ are the same.

\medskip

Our third theorem extends Lu's theorem towards the Bergman-incomplete situation. For simplicity, a domain $\Omega$ is said to satisfy Condition $(\star)$ if there exists some point $p \in \Omega$ such that 

1. $|K(z, p)|$ is bounded from above by a finite constant $\mathcal C>0$ for any $z \in \Omega$;

2. the Bergman representative coordinate $T $ defined at $p$ is continuous up to $ \overline\Omega$.

\begin{theorem} \label{2nd} Let $\Omega \subset \mathbb C^n$ be a bounded pseudoconvex  domain whose Bergman metric has its holomorphic sectional curvature identically equal to a negative constant $-c^2$.  
If  $\Omega$ satisfies Condition $(\star)$, then $\Omega$ is biholomorphic to the Euclidean ball $\mathbb B^n$ possibly less a relatively closed pluripolar set $E$ and $n = 2/c^2-1$.
 
\end{theorem}

\medskip

The pseudoconvexity in Theorem \ref{2nd} is a necessary assumption. For example, if we remove from $\mathbb B^n, n \geq 2$, a non-pluripolar compact subset $G$ of Lebesgue $\mathbb R^{2n}$-measure zero such that $\mathbb B^n \setminus G$ is connected, then by Hartogs' extension theorem the Bergman metric on $\mathbb B^n \setminus G$ extends to $\mathbb B^n$ so the assertion of Theorem \ref{2nd} fails.

\medskip

Based on Theorem \ref{biholo}, we prove Theorem \ref{2nd} in Section 4 by using a result  of 
Pflug and Zwonek \cite{PZ} on the so-called $L^2$-domain of holomorphy, 
which is the domain of existence of some $L^2$ holomorphic function. 
The boundary of a bounded $L^2$-domain of holomorphy contains no pluripolar part, so we get

\begin{cor} \label{Cor} Let $\Omega \subset \mathbb C^n$ be 
a bounded $L^2$-domain of holomorphy such that the holomorphic sectional curvature of the Bergman metric on $\Omega$ is identically equal to a negative constant $-c^2$. If  $\Omega$ satisfies Condition $(\star)$, then $\Omega$ is biholomorphic to the Euclidean ball $\mathbb B^n$ and $n = 2/c^2-1$.
\end{cor}

\section{Bergman potentials with self-bounded gradient}

For a bounded domain $\Omega \subset \mathbb C^n$, its Bergman kernel  is defined as 
$$
K (z, t):=\sum \varphi_j(z) \overline{\varphi_j(t)}, \quad  z, t  \in \Omega,
$$ 
where $\{\varphi_j\}_{j=1}^{\infty}$ is a complete orthonormal basis for the space of $L^2$ holomorphic functions. This definition does not depend on the choice of the basis.
One checks that $\log K (z, z)$ is smooth and strictly plurisubharmonic, and thus  defines the Bergman metric $g$ as
 \begin{equation} \label{metric} 
 \sum_{\alpha, \beta=1}^n g_{\alpha \bar \beta} (z) X_\alpha \overline{X_\beta} \equiv \sum_{\alpha, \beta=1}^n \frac{\partial^2\log K (z, z)}{\partial z_\alpha \partial \overline {z_\beta} }X_\alpha\overline{X_\beta},
 \end{equation} 
for $z \in \Omega$ and $X\in \mathbb C^n$. 

\medskip

 We first prove the following lemma which links the Bergman representative coordinate and the Bergman-Calabi diastasis.
 
 \begin{lem} \label{key} 
Let $\Omega  $ be a bounded domain in $\mathbb C^n$. 
For any $z_0 \in \Omega$, let $\Phi_{z_0}$ be the Bergman-Calabi diastasis relative to $z_0$ defined by \eqref{dia}.
Then,
 \begin{enumerate}

\item [1)]  the length of $\partial    \Phi_{z_0}$  measured by the Bergman metric   at any $p\in  \Omega \setminus A_{z_0}$  is 
\begin{equation} \label{=}
|\partial    \Phi_{z_0}|^2_g (p)=\sum_{\alpha, \beta=1}^n  w_\alpha(z_0)   g_{ \alpha \bar \beta } (p) \overline {w_\beta(z_0)},
\end{equation}
where $T (z) =(w_1, ... , w_n)^{\tau}$ is the Bergman representative coordinate defined by \eqref {rep}.

\item [2)]   for any $p \in \Omega $, 
 $T$ maps $ \Omega \setminus A_{p}$ to a ball of radius $R$
if and only if for any $z_0 \in \Omega $, 
the length of $\partial \Phi_{z_0}$  measured by the Bergman metric on $\Omega \setminus A_{z_0}$
is less than $R$. Here, $R$ is a positive constant depending only on $ \Omega$.
   \end{enumerate}

\end{lem} 

\begin{proof}
1)
Denote the complex gradient operators by $\nabla_z   := (\frac{\partial  }{\partial   {z_1} } , ... ,  \frac{\partial  }{\partial   {z_n} } )^{\tau}$ and $\nabla_{\bar{z}}   := (\frac{\partial  }{\partial   \bar{z_1} } , ... ,  \frac{\partial  }{\partial   \bar{z_n} } )^{\tau}$, where ${\tau}$ is the transpose of a matrix.
By the definition of the Bergman-Calabi diastasis, 
 $$
 -\nabla_{\bar{z}} \Phi_{z_0}  =   \nabla_{\bar{z}} \log \frac{ |K(z, z_0)|^2} {K(z, z)  }=    K(z_0, z)^{-1}  \nabla_{\bar{z}}  K(z_0, z) -  {\nabla_{\bar{z}}  \log K(z, z) }.
 $$
 In particular,  at any $p\in  \Omega \setminus A_{z_0}$,
 $$
 - \left. \nabla_{\bar{z}} \right|_{z=p} \Phi_{z_0}   =  K(z_0, p)^{-1}  \left. \nabla_{\bar{z}} \right|_{z=p} K(z_0, z) -  {\left. \nabla_{\bar{z}} \right|_{z=p}  \log K(z, z) }.
 $$
 Then  \eqref{rep} implies  that 
$$
T(z_0) = - [G^{-1}(p)]^{\tau}   {\left. \nabla_{\bar{z}} \right|_{z=p}  \Phi_{z_0}},
$$
where  $G:=(g_{ \alpha \bar \beta })$ and $[G^{-1}]^{\tau}= (g^{\bar j \alpha })$ is the inverse transpose  of $G$.
Therefore, 
 \begin{align*}
\sum_{\alpha, \beta=1}^n  w_\alpha(z_0)   g_{ \alpha \bar \beta } (p) \overline {w_\beta(z_0)} &= \overline {T(z_0)} ^{\tau} G(p) T(z_0)\\
&=  (  {\left. \nabla_{ z}  \right|_{z=p}\Phi_{z_0}})^{\tau} G(p)^{-1}  G(p) [G(p)^{-1}]^{\tau}  \left. \nabla_{\bar{z}} \right|_{z=p} \Phi_{z_0}  \\
 &=  (  {\left. \nabla_{ z}  \right|_{z=p}\Phi_{z_0}})^{\tau}   [G(p)^{-1}]^{\tau}  \left. \nabla_{\bar{z}} \right|_{z=p} \Phi_{z_0}\\
&= |\partial    \Phi_{z_0}|^2_{g}(p), 
 \end{align*}
which is the square of the  length of  $ {\nabla_{ z}  \Phi_{z_0}}$ 
measured by the Bergman metric   at $p$.

\medskip

 2) The direction of $ \Longrightarrow $. For any $z_0 \in \Omega$,
 from 1) we know that at any $p \in  \Omega \setminus A_{z_0}$, the square of the length of $\partial    \Phi_{z_0}$ measured by the Bergman metric is $  {
\sum_{\alpha, \beta=1}^n  w_\alpha(z_0)   g_{ \alpha \bar \beta } (p) \overline {w_\beta(z_0)}
}$, which is less than $R^2$ by assumption. 
Since $p$ is arbitrary,  
the length of $\partial \Phi_{z_0}$ 
is  less than $R$ on $\Omega \setminus A_{z_0}$.

\medskip

 The direction of $ \Longleftarrow $. For any $p \in \Omega$, take the Bergman representative coordinate $T$, which is defined on $ \Omega \setminus A_p$  by \eqref{rep}.
For any $z_0 \in \Omega \setminus A_p$, the left hand side of \eqref{=} is  less than $R^2$ since $p \in \Omega \setminus A_{z_0}$. So is the right hand side.
Therefore, $T$  maps $ \Omega \setminus A_{p}$ to a ball defined as $ \{(w_1, ... , w_n)^{\tau} :   \sum_{\alpha, \beta=1}^n  w_\alpha   g_{ \alpha \bar \beta } (p) \overline {w_\beta } <R^2 \}$.

 \end{proof}

 \noindent{}{\bf Remark.} If at $p$ the Bergman metric satisfies
\begin{equation} \label{nor}
  g_{\alpha \bar \beta} (p) =\delta_{\alpha  \beta },
\end{equation} then the right hand side of \eqref{=} reduces to $|T(z_0)|^2$, and the ball in 2) is a Euclidean ball $ \mathbb B^n$. 

\medskip

Moreover, we get the following lemma.

\begin{lem} \label{key!} 
Let $\Omega \subset \mathbb C^n$ be a bounded domain whose Bergman metric has its holomorphic sectional curvature  identically equal to a negative constant $-c^2$. 
At $p \in \Omega$, assume that   the Bergman metric satisfies
 \eqref{nor}
  and
take the Bergman representative coordinate $T (z) =(w_1, ... , w_n)$ 
defined by \eqref{rep}.
Then,
 \begin{enumerate}

\item [1)]   the Bergman-Calabi diastasis relative to $p$ can be written as
\begin{equation}  \label{Sch}
\Phi_{p}( z) =   {\frac{-2}{c^2}} \log      \left (1 - \frac{c^2}{2} |T(z)|^2 \right ), \quad z \in  \Omega \setminus A_p;
\end{equation}  
\item [2)]  $T$ maps $ \Omega $ to the ball $ \mathbb B^n := \{ w \in \mathbb C^n:   |w|^2 < {2}{c^{-2}} \}$.

 \end{enumerate}

 \end{lem}

\begin{proof}

At any $p \in \Omega$, by \cite{Lu}
there exists a neighbourhood $U_p$ such that the Bergman kernel can be locally decomposed as
\begin{equation}  \label{K(z, z)}
K(z, z)=  \left (1 - \frac{c^2}{2} |T(z)|^2 \right )^{\frac{-2}{c^2}} e^{f(T(z))+\overline{f(T(z))}}, \quad z\in U_p,
\end{equation}
where $f$ is holomorphic on $U_p$.
The map $T$ defined by \eqref{rep} is holomorphic on $\Omega \setminus A_p$, where $A_p:=\{z\in \Omega \, | \,K(z, p) =0 \, \}$ is the zero set of the Bergman kernel $K(\cdot , p)$.
Let  $\Omega^{\prime}:= \{z \in \Omega \setminus A_p : T(z) \in \mathbb B^n \}$ be the set of points in $\Omega \setminus A_p$ that are mapped into the ball.  In particular, $U_p \subset \Omega^{\prime}$.
By \eqref{K(z, z)} and the theory of power series, one may duplicate the variable with its conjugate so that  the full Bergman kernel can be complex analytically continued as
\begin{equation}  \label{K(z, z_0)}
K(z, {z_0})= \left (1 - \frac{c^2}{2} \sum_{\alpha=1}^n  w_\alpha(z)     \overline {w_\alpha(z_0)}   \right )^{\frac{-2}{c^2}}   e^{f(T(z))+\overline{f(T({z_0}))}}, \quad z, z_0\in U_p.
\end{equation}
Then for any $z_0\in U_p$,
 \begin{align*}
\Phi_{z_0}( z)&= \log \frac{\left (1 - \frac{c^2}{2} |T(z)|^2 \right )^{\frac{-2}{c^2}} e^{f(T(z))+\overline{f(T(z))}}   \left (1 - \frac{c^2}{2} |T(z_0)|^2 \right )^{\frac{-2}{c^2}} e^{f(T(z_0))+\overline{f(T(z_0))}}  }{  \left |1 - \frac{c^2}{2}   \sum_{\alpha=1}^n  w_\alpha(z)     \overline {w_\alpha(z_0)}   \right |^{\frac{-4}{c^2}} |e^{f(T(z))+\overline{f(T({z_0}))}}  |^2}\\
& =  {\frac{-2}{c^2}}   \log  \left[    \left (1 - \frac{c^2}{2} |T(z)|^2 \right )  \left (1 - \frac{c^2}{2} |T(z_0)|^2 \right )  \left  | 1 - \frac{c^2}{2}   \sum_{\alpha=1}^n  w_\alpha(z)     \overline {w_\alpha(z_0)}   \right  |^{ -2}  \right  ], \quad z\in U_p,
 \end{align*}
which yields that
$$
\Phi_{p}(z_0)= \Phi_{z_0}(p) = {\frac{-2}{c^2}} \log  \left (1 - \frac{c^2}{2} |T(z_0)|^2 \right ).
$$
 On the other hand, the Bergman-Calabi diastasis $\Phi_{p}(z)$ relative to $p$ is defined on $\Omega \setminus A_p$ and thus on $\Omega^{\prime}$, where ${\frac{-2}{c^2}} \log  \left (1 - \frac{c^2}{2} |T(z )|^2 \right )$ can be defined. Since these two real-analytic  functions coincide on  $U_p$, they are identical to each other on $\Omega^{\prime}$. That is, 
\begin{equation} \label{on Omega prime}
\Phi_{p}(z )=  {\frac{-2}{c^2}} \log  \left (1 - \frac{c^2}{2} |T(z )|^2 \right ), \quad z \in \Omega^{\prime}.
\end{equation} 

1) We claim that  no point in  $\Omega \setminus A_p$ is mapped outside the ball $\mathbb B^n$ by $T$. 

If not, suppose there exists some point $q \in  \Omega \setminus A_p$ that is mapped to $\{ w \in \mathbb C^n:   |w|^2 \geq {2}{c^{-2}} \}$. Choose some point $q_0 \in \Omega^{\prime}$.  Since $\Omega \setminus A_p$ is path-connected, one can choose a path $\gamma$ that connects $q_0$ and $q$. Suppose under $T$ the image of $\gamma$   intersects $\partial   \mathbb B^n$ firstly at some point $T(q_1)$.

 Along the path  $\gamma$ take a sequence of points $(q_l)_{l \in \mathbb N} \subset \Omega^{\prime}$ such that $q_l  \to q_1$. 
 Then by \eqref{on Omega prime},
$$
\Phi_{p}(q_l )=  {\frac{-2}{c^2}} \log  \left (1 - \frac{c^2}{2} |T(q_l)|^2 \right ).
$$
Here, as  $q_l  \to q_1$,  the left hand side is finite but  the right hand side  blows up to infinity.  This  is  a contradiction, so we have thus proved our claim, which says that $\Omega^{\prime}=  \Omega \setminus A_p $.
Therefore, \eqref{on Omega prime} in fact holds on $\Omega \setminus A_p $.

\medskip

2) Since $T$ maps $ \Omega \setminus A_p$ to the ball $ \mathbb B^n$ and satisfies 
\begin{equation} \label{bound}
|T(z)|^2<  {2}{ c^{-2}},
\end{equation}
by the Riemann removable singularity theorem, $T$ extends across the analytic variety $A_p$ to the whole domain $\Omega$ with $|T(z)|^2 \leq {2}{ c^{-2}}.$
  The maximum modulus principle yields  that \eqref{bound} in fact holds  on $\Omega$.

\end{proof}

 Lemma \ref{key!} implies Theorem \ref{zero}, part 1), which says that a bounded domain is Lu Qi-Keng if its Bergman metric has  constant holomorphic sectional curvature.

\begin{proof} [Proof of Theorem \ref{zero}, part 1)]
We first assume that at $p $ the Bergman metric satisfies \eqref{nor}. Let $A_{p} $ be the zero set  of the Bergman kernel $K(\cdot , p)$. Suppose
$A_{p} \neq \emptyset$.
Then take some point $q \in A_{p}$ and take a sequence of points $(z_j)_{j\in \mathbb N} \subset \Omega  \setminus A_p$ such that $z_j \to q$.
By \eqref{Sch},
$$
\Phi_{p}( z_j) =  {\frac{-2}{c^2}} \log      \left (1 - \frac{c^2}{2} |T(z_j)|^2 \right ).
$$
Letting $z_j  \to q$, we see that the above left hand side blows up to infinity, but  the right hand side is finite. This  is a contradiction, so $A_{p} = \emptyset$. Generally, for each  $p \in \Omega$, 
one performs a possible linear transformation $F$ from $\Omega$ to $\Omega_1$ such that the Bergman metric on $\Omega_1 $ at $F(p)$ satisfies 
\eqref{nor}. Since $F$ is a biholomorphism, the Bergman metric on $\Omega_1 $ also has constant holomorphic sectional curvature. 
Then, by the previous argument $K_{\Omega_1}(\cdot , F(p))$  has no zero set, so does $K(\cdot , p)$ due to the transformation rule of the Bergman kernel.

\end{proof}

Using Theorem \ref{zero}, part 1), and following the arguments of Bochner \cite{Bo47} and Lu \cite {Lu}, we prove
the remaining parts  of Theorem \ref{zero},  which can be seen as a generalization of Lemma \ref{key!}.

\begin{proof} [Proof of Theorem \ref{zero}, the remaining parts]
For simplicity, let $K=K(\cdot , \cdot)$ denote the Bergman kernel on $\Omega$.
The holomorphic sectional curvature of the Bergman metric   is defined as
 $$
 R_{\Omega} (z; X):=\left( \sum_{\alpha, \beta=1}^n g_{\alpha \bar \beta} X_\alpha X_{\bar \beta}\right)^{-2} \sum_{i, j, k, l=1}^n R_{i\bar j  k\bar l} \overline X_i X_j X_k \overline X_l, \quad z \in \Omega, \, X\in \mathbb C^n,
 $$
 where the curvature tensor is
 given by
\begin{align*}
R_{i\bar j  k\bar l}&=-\frac{\partial^4}{\partial w_i \partial \overline {w_j} \partial w_k \partial \overline {z_l} } \log K  + \sum_{\alpha, \beta=1}^n  g^{\bar \beta \alpha}  \frac{\partial^3 }{\partial w_i \partial   {w_k} \partial \overline {w_\beta} } \log K  \frac{\partial^3 }{\partial \overline {w_j} \partial \overline {w_l} \partial {w_\alpha} }  \log K \\
&=g_{i\bar j} g_{k\bar l} + g_{i\bar l} g_{k\bar j} - K^{-2} (K K_{i \overline {j} k  \overline {l} } - K_{i  k} K_{\bar{j} \bar l}) +
K^{-4} \sum _{\alpha, \beta=1}^n g^{\bar \beta \alpha}  (K K_{i k \overline {\beta}} - K_{i k } K_{\overline {\beta}}) (K K_{ \overline {jl} \alpha} - K_{ \overline {jl} } K_{\alpha } ).
\end{align*}
If the curvature is identically $-c^2$, then  (cf. \cite{Bo47, Hua54})
$$
R_{i\bar j  k\bar l}=\frac{-c^2}{2} (g_{i\bar j} g_{k\bar l} +g_{i\bar l} g_{k\bar j}),
$$
which implies that
\begin{equation} \label{const}
\frac{\partial^4}{\partial w_i \partial \overline {w_j} \partial w_k \partial \overline {w_l} } \log K= \sum_{\alpha, \beta=1}^n  g^{\bar \beta \alpha}  \frac{\partial^3 }{\partial w_i \partial   {w_k} \partial \overline {w_\beta} } \log K  \frac{\partial^3 }{\partial \overline {w_j} \partial \overline {w_l} \partial {w_\alpha} }  \log K  + \frac{c^2}{2} (g_{i\bar j} g_{k\bar l} +g_{i\bar l} g_{k\bar l}).
\end{equation}

Consider the test function 
 $$
\phi(w):= \frac{-2}{c^2} \log  \left (1 - \frac{c^2}{2} \sum_{\alpha, \beta=1}^n  w_\alpha   g_{ \alpha \bar \beta } (p) \overline {w_\beta } \right ) .
 $$
Then, $\phi$ induces a K\"{a}hler metric whose holomorphic sectional curvature is also identically equal to $-c^2$, with a similar identity as \eqref{const}.
Direct computations and \cite[Lemma 2]{Lu} yield  that
$$
\left. \frac{\partial^2 \log K }{\partial w_\alpha \partial \overline {w_\beta} } \right|_{w=0}=\left. \frac{\partial^2 \phi(w)}{\partial w_\alpha \partial \overline {w_\beta} } \right|_{w=0}=g_{ \alpha \bar \beta } (p),
$$
$$
\left. \frac{\partial^3 \log K }{\partial w_{\gamma} \partial w_\alpha \partial \overline {w_\beta} } \right|_{w=0}=\left. \frac{\partial^2 \phi(w)}{\partial w_{\gamma} \partial w_\alpha \partial \overline {w_\beta} } \right|_{w=0}=0.
$$
 The partial derivatives of order 4 can be computed directly from \eqref{const}; furthermore, the partial derivatives of higher order can be successively computed by taking all possible successive derivatives of \eqref{const}. As a result, they all vanish at $w=0$.
By the uniqueness of the Taylor expansion, it holds that
$$ 
K(z, z)=  \left (1 - \frac{c^2}{2} \sum_{\alpha, \beta=1}^n  w_\alpha(z)   g_{ \alpha \bar \beta } (p) \overline {w_\beta(z)} \} \right )^{\frac{-2}{c^2}} e^{F(T(z))+\overline{F(T(z))}}, \quad z\in U_p,
$$
for some holomorphic function $F$. By the theory of power series, one gets on $U_p$ the local formula 
\eqref{Lu_g}. 
Let  $\Omega^{\prime}:= \{z \in \Omega : \sum_{\alpha, \beta=1}^n  w_\alpha (z)   g_{ \alpha \bar \beta } (p) \overline {w_\beta } (z)  < {2}{c^{-2}} \}$ be the set of points in $\Omega$ that are mapped into the ball. 
Similar to the proof of Lemma \ref{key!}, one checks that \eqref{Lu_g} in fact hold true for each $ z \in \Omega^{\prime}$.
By the contradiction argument as demonstrated in the proof of Lemma \ref{key!}, we further observe that $\Omega^{\prime}=\Omega$ and thus have proved both  2) and 3).

\medskip

Part 4) then follows from 2) of Lemma \ref{key}.

\end{proof}

Part 4) of Theorem \ref{zero} says that the potential $\Phi_{z_0}$ has a self-bounded gradient  measured by the Bergman metric. On the Cartan classical domains, the first author and Li and Treuer in \cite{DLT} computed explicitly the (bounded) length of $\partial   \log K(z, z)$. Lee in \cite{L} studied K\"{a}hler-Einstein metrics admitting a global potential whose gradient has a constant length.  
   However, the potential $\log K(z, z)$ does not always have a self-bounded gradient in the Bergman metric on general bounded symmetric domains. An example of such a domain, which is biholomorphic to the bidisc, was constructed by Zimmer in \cite{Z}. 
 For convenience, define the following property.
 
 \begin{definition} A domain $\Omega \subset \mathbb C^n$ has Property ($\star\star$) if 
$$
|\partial \log K(z, z)|_{g} 
$$
is uniformly bounded on $\Omega$.
\end{definition}

Notice that Property ($\star\star$) is equivalent to: there exists $C > 0$ such that 
\begin{align}
\label{eq:property_b3_prime_equiv}
\abs{\partial \log K(z, z) (X)} \leq C\sqrt{g_{(z)}\left( X , \overline X\right)}
\end{align}
for all $X \in \mathbb C^n$ and $z \in \Omega$. Zimmer's example shows that Property ($\star\star$) is not invariant under biholomorphisms, and by imitating his construction we are able to prove

\begin{pro}\label{prop:not_inv} There exists a bounded domain $\Omega$ biholomorphic to $\mathbb B^2$ which does not have Property ($\star\star$). 
\end{pro}

\begin{proof} For a holomorphic function $\psi: \mathbb D \rightarrow \mathbb D-\{0\}$ define
\begin{align*}
F_\psi &: \mathbb B^2 \quad \to \quad \mathbb C^2\\
 &(z_1,z_2) \mapsto \left( \psi(z_2)z_1, z_2\right).
\end{align*}
Since $\psi$ is nowhere vanishing, $F$ is injective and hence is a biholomorphism onto its image. Let $\Omega_{\psi} := F_\psi(\mathbb B^2) \subset \mathbb B^2$. We claim that there exists some $\psi$ such that $\Omega_\psi$ does not have Property ($\star\star$). 

Notice that
\begin{align*}
F_\psi^\prime(z) = \begin{pmatrix} \psi(z_2) & \psi^\prime(z_2)z_1 \\  0 & 1 \end{pmatrix}.
\end{align*}
So $\det F_\psi^\prime(z_1, z_2) = \psi(z_2)$ and 
\begin{align*}
\abs{ \frac{\partial}{\partial z_2} \log \abs{\det F_\psi^\prime(z_1, z_2)}^2} = \abs{ \frac{\partial}{\partial z_2} \log \abs{\psi(z_2)}^2}=\abs{\frac{\psi^\prime(z_2)}{\psi(z_2)} }.
\end{align*}
Further, if $g$ is the Bergman metric on $\mathbb B^2$, then for $X:= (0, \frac{\partial}{\partial z_2}) \in \mathbb C^2$,
\begin{align*}
g_{(z_1,z_2)}\left( X, \overline X \right) =  \frac{ 1-\abs{z_1}^2 }{(1-\abs{z_1}^2- \abs{z_2}^2)^2}.
\end{align*}
So 
\begin{align*}
\frac{\abs{ \frac{\partial}{\partial z_2} \log \abs{\det F_\psi^\prime(z_1, z_2)}^2} }{\sqrt{g_{(z_1,z_2)}\left( X , \overline X \right)}}=\abs{\frac{\psi^\prime(z_2)}{\psi(z_2)} }  \frac{1-\abs{z_1}^2-\abs{z_2}^2} { \sqrt{1-\abs{z_1}^2 }}.
\end{align*}
In particular, putting $z_1=0$, we get
\begin{equation} \label{unbdd}
\frac{\abs{ \frac{\partial}{\partial z_2} \log \abs{\det F_\psi^\prime(0, z_2)}^2} }{\sqrt{g_{(0, z_2)}\left( X ,  \overline X \right)}}=\abs{\frac{\psi^\prime(z_2)}{\psi(z_2)} }   (1-\abs{z_2 }^2).
\end{equation}

Let $\psi : \mathbb D \rightarrow \mathbb D -\{0\}$ be a covering map. Then $\psi$ is a infinitesimial isometry  with respect  to the Poincar\'e metrics and so 
\begin{align*}
\frac{\abs{\psi^\prime(w)}}{2\abs{\psi(w)} \log \frac{1}{\abs{\psi(w)}}} = \frac{1}{1-\abs{w}^2}
\end{align*}
for all $w \in \mathbb D$. Then 
\begin{align*}
\frac{\abs{\psi^\prime(w)}}{\abs{\psi(w)}}\left(1-\abs{w}^2\right) =  2\log \frac{1}{\abs{\psi(w)}}
\end{align*}
is unbounded since $\psi(\mathbb D) = \mathbb D -\{0\}$. For this choice of $\psi$, either side of \eqref{unbdd} is unbounded so the domain $\Omega_\psi$ does not have Property ($\star\star$).

\end{proof}

Proposition \ref{prop:not_inv} says that under the constant negative holomorphic sectional curvature assumption, the bounded domain may not have Property ($\star\star$) in general.

    \section{Proof of Theorem \ref {biholo}}

A bounded domain $\Omega$ is said to be hyperconvex, if there exists a continuous plurisubharmonic function $\varphi $ such that the sublevel set $\{z\in \Omega: \varphi(z) <c\}$ is relatively compact in $\Omega$ for all $c<0$. Next, we will show  the implication (ii) $\Longrightarrow$ (iii) under the constant holomorphic sectional curvature assumption in Theorem \ref{biholo}.

\begin{proof} [Proof of Theorem \ref {biholo},  (ii) $\Longrightarrow$ (iii)]
By the definition of the Bergman kernel and the Cauchy-Schwarz inequality, the Bergman-Calabi diastasis $\Phi_{z_0} \geq 0$.
By 1) of Theorem \ref{zero}, $\Phi_{z_0}$ is non-negatively defined on $\Omega$.
We claim that the negative continuous function $\varphi:=-( 4^{-1}{c^{2}}  \Phi_{z_0} + 1)^{-1}$ is an exhaustion function for $\Omega$, i.e., the sublevel set $\{z\in \Omega: \varphi(z) <N\}$ is relatively compact in $\Omega$ for all $N<0$. If not, then there exists a point $w\in \partial \Omega \cap \{z\in \Omega:  4^{-1}{c^{2}}   \Phi_{z_0} + 1 < {-N^{-1}}\}$. Taking a sequence of points $(z_j)_{j \in \mathbb N}\subset  \Omega$ that tends to $w$, we know that $\lim_{\Omega \ni z_j\to w}   \Phi_{z_0} (z_j)  < {-4 c^{-2} (1+N^{-1})}  <+\infty$, which contradicts the fact that $\Phi_{z_0}$ blows up to infinity at $\partial \Omega $.
\medskip

To verify the plurisubharmonicity of $\varphi$ by Theorem \ref{zero}, we make the following computation
\begin{align*}
\partial \overline  \partial  \varphi  &= \partial \left(  (  4^{-1}{c^{2}} \Phi_{z_0} + 1)^{-2}  4^{-1}{c^{2}}  \overline  \partial \Phi_{z_0} \right)\\
&=-2 ( 4^{-1}{c^{2}}  \Phi_{z_0} + 1)^{-3}    4^{-1}{c^{2}}  \partial       \Phi_{z_0}    4^{-1}{c^{2}}  \overline  \partial  \Phi_{z_0}   + (  4^{-1}{c^{2}}   \Phi_{z_0} + 1)^{-2}   4^{-1}{c^{2}}  \partial \overline  \partial    \Phi_{z_0}  \\
&=( 4^{-1}{c^{2}}  \Phi_{z_0} + 1)^{-3} \left(  ( 4^{-1}{c^{2}} \Phi_{z_0} + 1) 4^{-1}{c^{2}}  g  -2^{-1}{c^{2}}  \partial    \Phi_{z_0}  (4 c^{-2})^{-1}  \overline  \partial      \Phi_{z_0}   \right)\\
 & \ge   (  4^{-1}{c^{2}}  \Phi_{z_0} + 1)^{-3} 2 (4 c^{-2})^{-2} \left(  2 c^{-2}  g  -  \partial   \Phi_{z_0}  \overline  \partial    \Phi_{z_0}   \right) \\
 & > 0,
\end{align*} 
which implies the hyperconvexity of $\Omega$.

\end{proof}

 \begin{proof} [Proof of Theorem \ref {biholo},  (iv) $\Longrightarrow$ (ii)]

It suffices to verify that for some fixed $p \in \Omega$, the Bergman-Calabi diastasis $\Phi_{p}( z)$ 
blows up to infinity at $\partial \Omega$.
After a possible linear transformation of domains, we may assume that \eqref{nor} holds at $p$.
By \eqref{Sch} and Theorem \ref{zero},
$$
\Phi_{p}( z) =  {\frac{-2}{c^2}} \log      \left (1 - \frac{c^2}{2} |T(z)|^2 \right ), \quad z \in  \Omega,
$$
where $T (z) =(w_1, ... , w_n)$ is the Bergman representative coordinate at $p$ defined by \eqref{rep}.

\medskip

By (iv), $T: \Omega \to \mathbb B^n$ is a biholomorphic map  that sends $p$ to $w=0$.
Moreover, the map $T$ is proper. As $z$ approaches $\partial \Omega$,
$\Phi_{p}( z)$ blows up to infinity uniformly as $T(z)$ approaches $\partial \mathbb B^n= \{w \in \mathbb C^n: |w|^2 = 2c^{-2} \}$. 
That is,  (iv) $\Longrightarrow$ (ii).

 \end{proof}

We also give a direct proof of $(ii) \Longrightarrow (iv)$ without relying on Lu's theorem.

 \begin{proof} [A direct proof of (ii) $\Longrightarrow $(iv)]
For $z_0\in \Omega$, after a possible linear transformation of domains, we may assume that \eqref{nor} holds at $z_0$.
By \eqref{Sch} and Theorem \ref{zero},
$$
\Phi_{z_0}( z) =  {\frac{-2}{c^2}} \log      \left (1 - \frac{c^2}{2} |T(z)|^2 \right ), \quad z \in  \Omega,
$$
where $T (z) =(w_1, ... , w_n)$ is the Bergman representative coordinate at $z_0$ defined by \eqref{rep}. Thus, $\Phi_{z_0}(z)$ blows up  to infinity if and only if $T(z)$ approaches $\partial \mathbb B^n= \{w \in \mathbb C^n: |w|^2 = 2c^{-2} \}$. Therefore, by Condition (ii) and Theorem \ref{zero}, $z$ approaches $\partial \Omega$ if and only if $T(z)$ approaches $\partial \mathbb B^n$. Consequently, the holomorphic map $T: \Omega \to \mathbb B^n$ is proper. 
 The Remmert proper mapping theorem then implies that the map $T: \Omega \to \mathbb B^n$ is a branched covering, and the branched points are exactly those at which the determinant of the complex Jacobian of $T$ vanishes.
 However, by our Proposition \ref {vol} the determinant $D_T(z)$ does not vanish, so  $T: \Omega \to \mathbb B^n$ is  an unbranched covering map. As $\mathbb B^n$ is simply connected, $T$ becomes a biholomorphism.

 \end{proof}

The rest of this section is devoted to the proof of the following proposition.

 \begin{pro}\label{annulus} There exist bounded domains in $\mathbb C^n$ which satisfy Condition (ii) but are not Bergman complete. 
\end{pro}
 
  \begin{proof} We first deal with the case of $n=1$. 
  Let $C_r:=\{z \in \mathbb C: r<|z|<1\}$, $0<r<1$, be an annulus.  
It is well-known that a bounded planar domain  $\Omega$ with $C^\infty$-smooth boundary is simply connected if and only if $K(z,w) \neq 0$ for all $z, w \in \Omega$.
 Thus, there exist $z_0, w_0 \in C_r$ such that $K(z_0, w_0) = 0$. Let $A_{z_0}:=\{z\in C_r \, | \,K(z, z_0) =0 \, \}  \neq  \emptyset$ be the zero set  of the Bergman kernel $K(\cdot , z_0)$.
 Since $A_{z_0}$ is an analytic variety, as domains $C_r \setminus A_{z_0}$ and $C_r$ have the same Bergman kernel and  metric.
 Then, the domain $C_r \setminus A_{z_0}$ is not Bergman complete.

 \medskip

We will show that $\Phi_{z_0}(z)$ blows up to infinity at $ \partial (C_r \setminus A_{z_0} ) =  A_{z_0} \cup \partial C_r.$
 To see this, for any point $q \in \setminus A_{z_0}$ and any sequence of points $\{z_j\}_{j \in \mathbb N} \subset C_r$ such that $z_j \to q$ as $j\to \infty$, it holds that $|K(z_j, z_0)|^2 \to 0$, so  
 $\Phi_{z_0}(z_j) \to \infty$. For any point $s \in \partial C_r$ and any sequence of points $\{s_j\}_{j \in \mathbb N} \subset C_r$ such that $s_j \to s$ as $j\to \infty$, the explicit formula of the Bergman kernel on $C_r$ implies that
 $|K(s_j, z_0)|^2$ remains finite, whereas $K(s_j, s_j) \to \infty$. Therefore,  $\Phi_{z_0}(s_j) \to \infty$ and consequently $C_r \setminus A_{z_0}$
 satisfies Condition (ii).

 \medskip
 
 By the product property of the Bergman kernel, higher dimensional examples can also be constructed by considering 
  $\mathbb B^{n-1} \times \{C_r \setminus A_{z_0}\} \subset \mathbb C^n$ when $n \geq 2$.
  
  \end{proof}

 \noindent{}{\bf Remark.} After the completion of this paper, we were kindly informed by Pflug that the Skwarczy\'nski invariant distance $ \rho_\Omega=\rho$ in \cite{Sk80}  is related to the Bergman-Calabi diastasis $\Phi_{z_0}$ in the following way:
$$
-2\log(1-\rho_\Omega ^2(z_0,z))=\Phi_{z_0}(z).
$$
Hence, Condition $(ii)$ is equivalent to the fact that $\rho_\Omega(z_0,z)\to 1$ if $z$
tends to $\partial\Omega$. He also pointed out that the implication $(iv)$ to $(ii)$ follows directly from a simple
calculation of $\rho_{\mathbb B^n}(0,z)$ showing that this function tends to $1$
if $z\to \partial {\mathbb B^n}$.

\section{Proofs of Theorem \ref {2nd} and Corollary \ref {Cor}}

     \begin{proof} [Proof of Theorem \ref {2nd}]

For any boundary point $q\in \partial \Omega$ such that 
 \begin{equation} \label{finite}
 \limsup \limits_ {\Omega \ni z\to q} K(z, z) < \infty,
 \end{equation}  
the results of Pflug \cite{P75} and Pflug and Zwonek \cite{PZ} say that $q \in \text{int} (\overline \Omega)$ and there exists a neighbourhood $U$ of $q$ such that $P:=U \setminus \Omega$ is a pluripolar set. 
Taking all boundary points $q_j$ that satisfy \eqref{finite}, we get the corresponding neighbourhoods $U_j$ and pluripolar sets $P_j$. Then the (bounded) domain
$$
\tilde\Omega:= \bigcup\limits_{j}U_j \cup \Omega
 $$
has the same Lebesgue measure as $\Omega$ due to the pluripolarity. Moreover, 
$\partial \tilde \Omega$ coincides with the non-pluripolar part of $\partial \Omega$ in view of \cite{PZ}.

\medskip

Let $w \in \partial \Omega$ be an arbitrary boundary point  such that 
$$
\limsup_{z \to w} K(z, z) = \infty.
$$
By the assumption Condition $(\star)$, it follows  that
  \begin{equation}  \label{limsup K}
   \limsup \limits_ {\Omega \ni z\to w} \Phi_{p}( z) \geq \limsup \limits_ {\Omega \ni z\to w} \Phi_{p}( z) + \limsup \limits_ {\Omega \ni z\to w}\log \frac{|K(z, p)|^2  }{ \mathcal C ^2  } =   \limsup \limits_ {\Omega \ni z\to w}\log \frac{K(z, z) K(p, p)} { \mathcal C ^2} = \infty.
  \end{equation} 
 Since the Bergman representative coordinate $T (z) $ at $p$ defined by \eqref{rep} is continuous up to $ \overline\Omega$, for any two sequences of points $ (z_j)_{j\in \mathbb N}, (w_j)_{j\in \mathbb N} \subset \Omega$, both approaching $w$, it holds that
  \begin{equation}  \label{same}
\lim_{j \to \infty}T(z_j)=\lim_{j \to \infty}T(w_j).
  \end{equation} 
After a possible linear transformation of domains we may assume that \eqref{nor} holds at $p$.
By Theorem \ref{zero}, the Bergman-Calabi diastasis $\Phi_{p}$ relative to $p$ can be written as 
  \begin{equation}  \label{formula}
\Phi_{p}( z) =  {\frac{-2}{c^2}} \log      \left (1 - \frac{c^2}{2} |T(z)|^2 \right ), \quad z \in  \Omega.
  \end{equation} 
Therefore, \eqref{limsup K} further implies 
 $$
\lim \limits_ {\Omega \ni z\to w} \Phi_{p}( z)= \infty,
$$
  in view of \eqref {same} and \eqref{formula}.
Since $w$ is arbitrary, we know that the Bergman-Calabi diastasis $\Phi_{p}$  blows up to infinity at all points in $\partial \tilde \Omega$.

\medskip

Next, we will see that as domains $\tilde\Omega$ and $\Omega$ have the same Bergman metric.
Notice that $P_j=U_j \cap \partial \Omega$ is relatively closed in $U_j$. 
Restrict any function $f \in L^2 \cap \mathcal O (\Omega)$ to $U_j\setminus P_j$. By a result \cite{S82} of Siciak there exists a function $F \in L^2 \cap \mathcal O (U_j)$ such that $F=f$ on $U_j\setminus P_j$. Hence we get an
$L^2$ holomorphic extension to $\Omega \cup U_j$. By the same procedure, we extend $f$ to $\tilde\Omega$, whose Bergman metric also has constant holomorphic sectional curvature.
Then, Theorem \ref{biholo} guarantees that $\tilde\Omega$ is biholomorphic to the Euclidean ball $\mathbb B^n$ and $n = 2/c^2-1$. 
Define the set 
$$
E:=\bigcup\limits_{j}P_j = \bigcup\limits_{j} U_j \cap \partial \Omega = \tilde\Omega \cap \partial \Omega,
$$ 
which is  relatively closed in $\tilde\Omega$.
Since the biholomorphic (pre)images and countable union of pluripolar sets are still pluripolar, the proof of Theorem \ref {2nd} is now complete.

 \end {proof}

We will use Theorem \ref {2nd} to prove Corollary \ref {Cor} as follows.

 \begin{proof} [Proof of Corollary \ref {Cor}]
An $L^2$-domain of holomorphy is pseudoconvex, and the boundary of a bounded  $L^2$-domain of holomorphy contains no pluripolar part, cf. \cite{PZ, I04}.
By Theorem \ref {2nd}, the possible pluripolar set $E
$ is empty and the domain  is biholomorphic to the ball. 

  \end {proof}

\section{Results on the exhaustiveness}

We say a domain $\Omega$ is Bergman exhaustive if at any $z_0 \in \partial \Omega$,  $$
\lim_{\Omega\ni z\to z_0} K (z, z) = \infty.
$$ 
Under the constant holomorphic sectional curvature assumption in Theorem \ref{biholo}, Condition (ii) implies  the Bergman exhaustiveness  by Proposition \ref{Mok-Yau} below.

 \begin{pro}  \label{Mok-Yau}
Let $D$ be a bounded homogeneous domain in $\mathbb C^n, n\geq 1$.
Then the Bergman kernel $K$ on $D$ satisfies
$$
K(z, z)  \ge \frac{C}{\delta_{D}(z)^2 (-\log \delta_{D}(z))^2}, \quad z\in D,
$$ 
where $\delta_{D}$ is the Euclidean distance function to boundary of $D$ and $C>0$.
In particular, $D$ is Bergman exhaustive.
 \end {pro}

 \begin{proof} 

 The Bergman metric on $D$ is complete, cf. \cite{Lu}.
By \eqref{metric}, let $\mathcal V:=\det (g_{j \bar k}) $ denote the Bergman volume form,
and let $V $ be its coefficient in the Euclidean coordinate such that $\mathcal V= V (\frac{i}{2})^n dz_1 \wedge \bar{dz_1} \wedge ... \wedge dz_n \wedge \bar{dz_n}$.
Consider the function $B(z):=V (z) / K(z, z)$. Due to the transformation formula of the Bergman kernel, $B$ is invariant under biholomorphic mappings in the sense that if $\phi$ is a biholomorphic map
 from $D$ to $\Omega$, then $B_{D}(z)=B_{\Omega}(\phi(z))$. 
For any two points $p, q$ in $D$, which is  homogeneous, there exists a biholomorphic self-map $\varphi$ on $D$ such that $\varphi (p) =q$. So $B(p)=B(q)$ and thus $B $ is a constant function on $D$. From this, one easily sees that the 
 Bergman metric on $D$ is a (complete) K\"{a}hler-Einstein metric.
Since $D$ is pseudoconvex, by the argument of Mok and Yau in \cite{MY83}, we know that
$$
V  \ge \frac{C}{\delta_{D}^2 (-\log \delta_{D})^2},
$$
which implies the desired lower bound estimate for $K$. Moreover, as $z$ tends to the boundary of $D$, $K(z, z)$ blows up to infinity.

\end {proof}

Proposition \ref{Mok-Yau} particularly implies the Bergman exhaustiveness of a bounded domain that is biholomorphic to a ball. See also \cite[Proposition 5.2] {K} for a proof of the fact that a bounded homogeneous domain is Bergman exhaustive. In general, the Bergman exhaustiveness is not biholomorphically invariant. For instance, the Hartogs triangle $\mathcal H $ is Bergman exhaustive, but its  biholomorphic image $\mathbb D \times \mathbb D^* $ is neither homogeneous nor Bergman exhaustive.

\medskip

A general problem raised by Yau \cite[pp. 679]{Yau} is to characterize manifolds whose Bergman metrics are K\"{a}hler-Einstein. 
Our results in Section 1 can be viewed as a particular case of Yau's problem of which the Bergman metric is of constant negative holomorphic sectional curvature.
When the manifold is a smoothly bounded strictly pseudoconvex domain, Cheng conjectured that the Bergman metric is K\"{a}hler-Einstein if and only if the domain is biholomorphic to the ball. 
After the previous works of Fu and the second author \cite{FW} and Nemirovski and Shafikov \cite{NS}, the Cheng conjecture was then confirmed by Huang and Xiao in \cite{HX}. 
Recently in his 2021 thesis \cite{M}, Alec Martin used the Bergman invariant function to characterize the CR-spherical boundary of a strongly pseudoconvex domain, and thus gave somehow an alternative argument in the final step of proving Cheng's conjecture.
On the other hand, Kohn's subelliptic estimate in his theory of 
the $\bar \partial$-Neumann problem \cite{FK} was applied by Kerzman in \cite{Ker} to show that on a bounded strictly pseudoconvex domain $\Omega$ with
$C^\infty$-smooth boundary, for each fixed $z_0 \in \Omega$, the Bergman kernel $K(\cdot, z_0)$ is $C^\infty$ up to the boundary. He also gave in \cite[p.151-152]{Ker} an example of a simply-connected planar domain $D$ whose Bergman kernel $K(\cdot, z_0)$ blows up to infinity at $\partial D$.
One can check from the following formula in Proposition \ref{vol} that in Kerzman's example the determinant of the Jacobian of a Bergman representative coordinate is unbounded.

\medskip

Our next result gives an explicit formula of the 
Bergman (K\"{a}hler-Einstein) volume form.

 \begin{pro}  \label{vol}
 Let $\Omega \subset \mathbb C^n$ be a bounded domain  whose Bergman metric $g$ has its holomorphic sectional curvature  identically equal to a negative constant $-c^2$. 
 Let $V$ be the coefficient in the Euclidean coordinate of the Bergman volume form $\mathcal  V:=\det (g_{\alpha \bar \beta} ) $ such that $\mathcal V= V (\frac{i}{2})^n dz_1 \wedge \bar{dz_1} \wedge ... \wedge dz_n \wedge \bar{dz_n}$. Then,
 \begin{equation} \label{V}
V (z)=  |D_T(z)|^2 \left (1 - \frac{c^2}{2} \sum_{\alpha, \beta=1}^n  w_\alpha   g_{ \alpha \bar \beta } (p) \overline {w_\beta } \right )^{-n-1},
 \end{equation}
where $ T (z) =(w_1, ... , w_n)$ is the Bergman representative coordinate at $p \in \Omega$ and $D_T(z)$ is the determinant of the complex Jacobian of $T$. In particular, $D_T(z)$ does not vanish on $\Omega$.

 \end {pro}

\begin {proof}

For simplicity, assume that \eqref{nor} holds at $p$.
By \eqref{Sch} and Theorem \ref{zero}, part 1), the Bergman-Calabi diastasis 
$\Phi_{p}$ relative to $p$ can be written as 
$$
\Phi_{p}( z) =  {\frac{-2}{c^2}} \log      \left (1 - \frac{c^2}{2} |T(z)|^2 \right ), \quad z \in  \Omega,
$$

Direct computations yield that
 \begin{align*}
g_{\alpha \bar \beta} (z)  &= \frac{\partial^2 \Phi_{p}( z) }{\partial z_\alpha \partial \overline {z_\beta} }\\
&=\sum_{i, j =1}^n  \left (1 - \frac{c^2}{2} |T(z)|^2 \right )^{-2}   \left[    \delta_{i j}  \left (1 - \frac{c^2}{2} |T(z)|^2 \right )     + \frac{c^2}{2}      \overline{w_i(z)}    w_j(z)    \right] {\frac{\partial w_i(z)}  {\partial {z_\alpha}} }  \frac{\partial  \overline{w_j (z)} }{   \partial  {\overline {z_\beta}} } \\
&= \sum_{i, j =1}^n    T_{i \bar j} (w)  {\frac{\partial w_i(z)}  {\partial {z_\alpha}} }  \frac{\partial  \overline{w_j(z)} }{   \partial  {\overline {z_\beta}} }, \quad z\in \Omega,
\end{align*}
where
$$
T_{i \bar j} (w) :=   \frac{\partial^2 }{\partial w_i \partial \overline {w_j} } \left(  \frac{-2}{c^2} \log (1 - \frac{c^2}{2} |w|^2) \right)
$$
gives a K\"{a}hler metric on $ \mathbb B^n$.
Therefore,
$$
V (z)=\det (g_{\alpha \bar \beta}) = |D_T(z)|^2 \det (T_{i \bar j})  = |D_T(z)|^2 \left (1 - \frac{c^2}{2} |T(z)|^2 \right )^{-n-1}.
$$
For general $p$, one may use the formula \eqref{Lu_g} instead of \eqref{Sch} to get \eqref{V}, which yields that the branch locus is empty.

   \end {proof}

As a final remark, our method also yields the following statement easily.

\begin{theorem}
 
Let $\Omega \subset \mathbb C^n$ be a bounded domain whose Bergman metric has its holomorphic sectional curvature  identically equal to a negative constant $-c^2$. Assume that $\Omega$ is Bergman exhaustive and
there exists some point $p \in \Omega$ such that 
 $|K(\cdot, p)|$ is bounded from above on $\Omega$ by a finite constant.  
 Then, $\Omega$ is biholomorphic to the Euclidean ball $\mathbb B^n$ and $n = 2/c^2-1$.

\end{theorem}

\subsection*{Funding}

{\fontsize{11.5}{10}\selectfont

The research of the first author is supported by AMS-Simons travel grant. 
}

\subsection*{Acknowledgements} 
{\fontsize{11.5}{10}\selectfont
The first author sincerely thanks Professors Peter Pflug and Ming Xiao for related discussions.}

\bibliographystyle{alphaspecial}

\fontsize{11}{9}\selectfont

\vspace{0.5cm}

\noindent dong@uconn.edu, 

\vspace{0.2 cm}

\noindent Department of Mathematics, University of Connecticut, Storrs, CT 06269-1009, USA

\vspace{0.4cm}

\noindent wong@math.ucr.edu,

\vspace{0.2 cm}

\noindent Department of Mathematics, University of California, Riverside, CA 92521-0429, USA

 \end{document}